\documentclass[reqno,12pt]{amsart}

\usepackage{amsfonts,amsmath,amsthm,amssymb,amscd}
\usepackage[all]{xy}
\title{Dualities for absolute zeta functions and multiple gamma functions}
\author{Nobushige Kurokawa}
\address{Nobushige Kurokawara\\ Tokyo Institute of Technology}
\email{kurokawa@math.titech.ac.jp}
\author{Hiroyuki Ochiai}
\address{Hiroyuki Ochiai\\ Kyushu University}
\email{ochiai@math.kyushu-u.ac.jp}

\usepackage[mathscr]{eucal}
\usepackage{mathrsfs}
\begin{document}


\newcommand{\ff}{\mathbb{F}}
\newcommand{\fone}{{\mathbb{F}_1}}
\newcommand{\fp}{{\mathbb{F}_p}}
\newcommand{\C}{\mathbb{C}}
\newcommand{\R}{\mathbb{R}}
\newcommand{\Z}{\mathbb{Z}}

\newcommand{\Real}{\operatorname{Re}}
\newcommand{\Spec}{\operatorname{Spec}}
\newcommand{\rank}{\operatorname{rank}}

\newcommand{\Gm}{\mathbb{G}_{\rm m}}
\newcommand{\Gmr}{\Gm^{\otimes r}}
\newcommand{\uo}{\underline{\omega}}
\newcommand{\SLr}{\mathbb{SL}(r)}
\newcommand{\GLr}{\mathbb{GL}(r)}

\theoremstyle{plain}
\newtheorem*{ThmA}{Theorem A}
\newtheorem*{ThmB}{Theorem B}
\newtheorem*{ThmC}{Theorem C}
\newtheorem{Thm}{Theorem}
\newtheorem*{Lem}
{Lemma}
\theoremstyle{definition}
\newtheorem{Dfn}{Definition}
\newtheorem{Ex}{Example}

\maketitle
\begin{abstract}
We study absolute zeta functions from the
view point of a canonical normalization.
We introduce the absolute Hurwitz zeta function for the 
normalization.  In particular, we show that the theory of
multiple gamma and sine functions
gives good normalizations in cases related to
the Kurokawa tensor product.  In these cases,
the functional equation of the absolute zeta function turns out to be
equivalent to the simplicity of the associated non-classical
multiple sine function of negative degree.
\end{abstract}

\section{Introduction}
The absolute zeta function of a scheme $X$ over
$\fone$ was first studied by Soul\'e \cite{Soule}
as a ``limit of $p\to 1$'' of the (congruence) zeta function
over $\fp$:
see Kurokawa \cite{K2} and Deitmar \cite{D} also.
Then, Connes and Consani \cite{CC1} \cite{CC2}
investigated the absolute zeta function as the following integral
\begin{equation*}
\zeta_X(s) = \exp\left( 
\mbox{$\displaystyle\int_1^\infty \frac{N_X(u)}{u^{s+1} \log u} du$} \right),
\end{equation*}
where
\begin{equation*}
N_X(u) = \left| X(\ff_{1^{u-1}}) \right|
\end{equation*}
is a suitably interpolated ``counting function.''
Here we must pay attention to the needed normalization for the integral near $u=1$:
see \cite{CC1} \cite{CC2} for a discussion.
In \cite[Theorem~4.13]{CC1} \cite[Theorem~4.3]{CC2} 
Connes and Consani calculated $\zeta_X(s)$ for Noetherian schemes
via the Kurokawa tensor product of \cite{K1}.

Our purpose is to introduce the absolute Hurwitz zeta function
\begin{equation*}
Z_X(w;s) = \frac{1}{\Gamma(w)} \int_1^\infty
\frac{N_X(u)}{u^{s+1} (\log u)^{1-w}} du
\end{equation*}
to get the canonical normalization:
\begin{equation*}
\zeta_X(s) =\exp \left(
\left.\frac{\partial}{\partial w} Z_X(w;s)\right|_{w=0}
\right).
\end{equation*}

This normalization is essentially due to Riemann (1859)
and it is used in the theory of multiple gamma and sine function as follows.

For each integer $r \ge 1$,
the $r$-ple Hurwitz zeta function $\zeta_r(w;x)$ is defined in 
$\Real(w) >r$ as
\begin{equation}
\zeta_r(w;x) = \sum_{n=0}^\infty {}_r H_n (n+x)^{-w}
\end{equation}
where $_r H_n ={{n+r-1}\choose{n}}$.

The analytic continuation of $\zeta_r(w;x)$ to all $w \in \C$
is obtained via the integral representation of Riemann
\begin{align*}
\zeta_r(w;x) 
&= \frac{1}{\Gamma(w)} \int_0^\infty (1-e^{-t})^{-r} e^{-xt} t^{w-1} dt \\
&= \frac{1}{\Gamma(w)} \int_1^\infty (1-u^{-1})^{-r} u^{x-1} (\log u)^{w-1} du
\end{align*}
by treating the integral around $u=1$ in the usual way.

Thus, by using such analytic continuation we get the $r$-ple gamma function
\begin{equation*}
\Gamma_r(x) =\exp\left(
\left.\frac{\partial}{\partial w} \zeta_r(w;x)\right|_{w=0}
\right)
\end{equation*}
and the $r$-ple sine function
\begin{equation*}
S_r(x) = \Gamma_r(x)^{-1} \Gamma_r(r-x)^{(-1)^r}.
\end{equation*}
We refer to Barnes \cite{B} (1904)
and Kurokawa-Koyama \cite{KK} (2003) for details,
where more general multiple gamma functions and multiple
sine functions were treated respectively.

We report three results in this introduction.
First, for a function $N:(1,\infty) \rightarrow \C$
we use
\begin{equation*}
Z_N(w;s) = \frac{1}{\Gamma(w)} 
\int_1^\infty N(u) u^{-s-1} (\log u)^{w-1} du
\end{equation*}
and
\begin{equation*}
\zeta_N(s) =\exp\left(
\left.\frac{\partial}{\partial w} Z_N(w;s)\right|_{w=0}
\right)
\end{equation*}
also.

\begin{ThmA}
Let $N(u) = \displaystyle\sum_{\alpha} m(\alpha) u^\alpha$ be a finite sum.
Then:
\begin{itemize}
\item[(1)] $Z_N(w;s) = \displaystyle \sum_{\alpha} m(\alpha) (s-\alpha)^{-w}$.
\item[(2)] $\zeta_N(s) = \displaystyle \prod_{\alpha} (s-\alpha)^{-m(\alpha)}$.
\end{itemize}
\end{ThmA}

This result is applicable to calculate many examples
(see \cite{K2}) of absolute zeta functions under our canonical normalization.
We note two simple examples.
\begin{Ex}
Let $X=\Spec \fone$.
Then
\begin{align*}
&N_X(u) =1,\\
&Z_X(w;s) = s^{-w},\\
& \zeta_X(s)=1/s.
\end{align*}
\end{Ex}
\begin{Ex}
Let $X=\mathbb{SL}(2)$.
Then 
\begin{align*}
&N_X(u) = u^3-u,\\
&Z_X(w;s) = (s-3)^{-w}  (s-1)^{-w},\\
& \zeta_X(s) = (s-1)/(s-3).
\end{align*}
\end{Ex}

Now the following result shows a functoriality.
\begin{ThmB}
\begin{itemize}
\item[(1)]
For $N_1, N_2 : (1,\infty) \rightarrow \C$
let
\begin{equation*}
(N_1 \oplus N_2) (u) = N_1(u) + N_2(u).
\end{equation*}
Then
\begin{equation*}
Z_{N_1 \oplus N_2}(w;s) = Z_{N_1}(w;s) + Z_{N_2}(w;s)
\end{equation*}
and
\begin{equation*}
\zeta_{N_1\oplus N_2}(s) = \zeta_{N_1}(s) \zeta_{N_2}(s).
\end{equation*}
\item[(2)]
Let
\[
N_i(u) = \sum_{\alpha_i} m_i(\alpha_i) u^{\alpha_i}
\]
for $i=1,2$.
Suppose that both are finite sums.
Put
\[
(N_1\otimes N_2)(u) = N_1(u) N_2(u).
\]
Then
\begin{align*}
& Z_{N_1 \otimes N_2}(w;s) \\
& = \sum_{\alpha_1,,\alpha_2}
m_1(\alpha_1) m_2(\alpha_2) (s-(\alpha_1+\alpha_2))^{-w}
\end{align*}
and
\[
\zeta_{N_1\otimes N_2}(s) = \prod_{\alpha_1, \alpha_2}(s-(\alpha_1+\alpha_2))^{-m_1(\alpha_2) m_2(\alpha_2)}.
\]
\end{itemize}

\end{ThmB}

This tensor product is essentially the Kurokawa tensor product
originated in \cite{K1}
(see \cite{M}, \cite{CC1} and \cite{CC2})
when $\alpha_j$'s are real.
We remark that for general $N_j$'s
(``infinite sums'' or ``generalized functions'')
we must resolve various difficulties.

For the next result we notice that our construction of $\zeta_r(w;x)$, $\Gamma_r(x)$
and $S_r(x)$ is valid for negative $r$ also
(see the later explanation).

\begin{ThmC}
Let $r$ be a positive integer.
Then
\begin{itemize}
\item[(1)]
$\displaystyle Z_{\Gm^{\otimes r}}(w;s) = \zeta_{-r}(w;s-r)$.
\item[(2)]
$\displaystyle\zeta_{\Gmr}(s) = \Gamma_{-r}(s-r)$\\
$\displaystyle= \prod_{j=1}^r(s-j)^{(-1)^{r-j-1}{r \choose j}}$ \\
$\displaystyle= \left( (1-1/s)^{\otimes r} \right)^{-1}$,\\
where $\otimes r$ is the Kurokawa tensor product.
\item[(3)]
We have the functional equation
\begin{equation*}
\zeta_{\Gmr}(s) = \zeta_{\Gmr}(r-s)^{(-1)^r},
\end{equation*}
which is equivalent to $S_{-r}(x)=1$.
\end{itemize}
\end{ThmC}

Our result would suggest that
\[
\zeta_{\Gmr}(s) = \Gamma_{-r}(s-r)
\]
holds for $r<0$ also with the functional equation $s \leftrightarrow -r-s$.
For example
\[
\zeta_{\Gm^{\otimes-1}}(s) 
= \Gamma_1(s+1)
= \frac{\Gamma(s+1)}{\sqrt{2\pi}}
\]
and the functional equation $s \leftrightarrow 1-s$
is the reflection formula of Euler:
\[
\Gamma_1(s+1) \Gamma(2-s)
= S_1(s+1)^{-1} = - \frac{1}{2 \sin(\pi s)}.
\]
We remark that Manin \cite[\S1.7]{M}
indicated an idea to consider the gamma function
as the zeta function of the ``dual infinite dimensional projective space over $\fone$.''

\section{Multiple gamma functions and multiple sine functions}
We recall the construction of the multiple Hurwitz zeta function:
\begin{align*}
\zeta_r(w;x)
&= \sum_{n=0}^\infty {n+r-1 \choose n} (n+x)^{-w} \\
&= \frac{1}{\Gamma(w)} \int_0^\infty (1-e^{-t})^{-r} e^{-xt} t^{w-1} dt \\
&= \frac{1}{\Gamma(w)} \int_1^\infty (1-u^{-1})^{-r} u^{-x-1} (\log u)^{w-1} du.
\end{align*}
This definition is valid for any $r \in \R$ with sufficiently large $\Real(x)$ and $\Real(w)$,
so we have the analytic continuation to all $w \in \C$ via the usual method. Thus, we get
\begin{equation*}
\Gamma_r(x) = \exp\left(\left. \frac{\partial}{\partial w} \zeta_r(w;x) \right\vert_{w=0} \right)
\end{equation*}
and 
\begin{equation*}
S_r(x) = \Gamma_r(x)^{-1} \Gamma_r(r-x)^{(-1)^r}
\end{equation*}
for any $r \in \R$
(or $r \in \Z$ at least 
without ambiguity of the meaning of $(-1)^r$).
For readers interested in the theory of $r<0$, 
we refer to \cite{KO}.

\begin{Thm}
Let $r$ be a negative integer.
Then
\begin{itemize}
\item[(1)]
$\displaystyle\Gamma_r(x) =
\prod_{n=0}^{-r} (x+n)^{(-1)^{n+1} {-r \choose n}}$. 

\item[(2)]
$S_r(x)=1$.
\end{itemize}
\end{Thm}
\begin{proof}
We have
\begin{align*}
\zeta_r(w;x)
&= \sum_{n=0}^\infty {n+r-1 \choose n} (n+x)^{-w}\\
&= \sum_{n=0}^\infty (-1)^n {-r \choose n} (n+x)^{-w}.
\end{align*}
Hence
\begin{align*}
\Gamma_r(x) 
&=\exp\left(\sum_{n=0}^{-r} (-1)^{n+1} {-r \choose n} \log(n+x) \right)\\
&= \prod_{n=0}^{-r} (n+x)^{(-1)^{n+1} {-r \choose n}}.
\end{align*}
Next,
\begin{align*}
&S_r(x) 
= \Gamma_r(x)^{-1} \Gamma_r(r-x)^{(-1)^r}\\
&= \prod_{n=0}^{-r} (n+x)^{(-1)^n {-r \choose n}}
\times \prod_{n=0}^{-r} (n+r-x)^{(-1)^{n-r+1} {-r \choose n}} \\
&= \prod_{n=0}^{-r} (n+x)^{(-1)^n {-r \choose n}}\\
& \qquad \times \prod_{n=0}^{-r} ((-r-n)+x)^{(-1)^{(-r-n)+1} {-r \choose n}}, 
\end{align*}
where we used
\begin{equation*}
\sum_{n=0}^{-r} (-1)^n {-r \choose n} =0.
\end{equation*}
Hence
\begin{align*}
S_r(x) &= 
\prod_{n=0}^{-r} (n+x)^{(-1)^n {-r \choose n}}
\times \prod_{n=0}^{-r} (n+x)^{(-1)^{n+1} {-r \choose n}} \\
&=1.
\end{align*}
\end{proof}

This result can be generalized to the multi-period case
$\uo=(\omega_1,\ldots, \omega_r)$ with $\omega_1,\ldots, \omega_r>0$
as follows,
where the above case is contained as $\uo=(1,\ldots,1)$.
Put
\begin{align*}
&\zeta_{-r}(w;x,\uo) \\
&\qquad= 
\sum_{1\le i_1<\cdots < i_k \le r} (-1)^k (x+\omega_{i_1}+\cdots+\omega_{i_k})^{-w},
\\
&\Gamma_{-r}(w,\uo)
= \exp\left(\left. \frac{\partial}{\partial w} \zeta_{-r}(w;x,\uo) \right|_{w=0} \right),
\intertext{and}
&S_{-r}(x, \uo) 
=\Gamma_{-r}(x,\uo)^{-1}\\
&\quad\qquad\times \Gamma_{-r}(-(\omega_1+\cdots+\omega_r)-x, \uo)^{(-1)^r}.
\end{align*}
Then we have (see \cite{KO} for more generalizations also)
\begin{align*}
&\zeta_{-r}(w;x,\uo) \\
&\quad = \frac{1}{\Gamma(w)} \int_0^w 
(1-e^{-t \omega_1})\cdots(1-e^{-t \omega_r}) e^{-xt} t^{w-1} dt, \\
&\Gamma_{-r}(x,\uo) 
= \prod_{1\le i_1<\cdots < i_k \le r} (x+\omega_{i_1}+\cdots+\omega_{i_k})^{(-1)^k},
\\
\intertext{and}
&S_{-r}(x,\uo) =1.
\end{align*}
For example, we get
\[
\zeta_{\mathbb{SL}(2)}(s) = \Gamma_{-1}(s-3,2) = \frac{s-1}{s-3}.
\]
More generally:
\begin{align*}
\zeta_{\mathbb{SL}(r)}(s)  = \Gamma_{-(r-1)}(s-(r^2-1), (2,3,\cdots,r)) \\
\intertext{and}
\zeta_{\mathbb{GL}(r)}(s)  = \Gamma_{-r}(s-r^2, (1,2,3,\cdots,r)),
\end{align*}
where
$
\left\{
\begin{array}{l}
r-1 = \rank \mathbb{SL}(r)\\
r^2-1 = \dim \mathbb{SL}(r)
\end{array}
\right.$
and
$
\left\{
\begin{array}{l}
r = \rank \mathbb{GL}(r)\\
r^2 = \dim \mathbb{GL}(r).
\end{array}
\right.$
We obtain the functional equations
\begin{align*}
& \zeta_{\mathbb{SL}(r)}(s) = \zeta_{\SLr}(r(3r-1)/2-1-s)^{(-1)^{r-1}}, \\
\intertext{and}
& \zeta_{\mathbb{GL}(r)}(s) = \zeta_{\GLr}(r(3r-1)/2-s)^{(-1)^{r}}
\end{align*}
from the triviality of the multiple sine function of negative order exactly similar to Theorem~C.

\begin{Thm}
Let $r$ be a negative real number.
Then:
\begin{itemize}
\item[(1)]
$\zeta_r(m;x) =0$
for each integer $m$ satisfying $r<m\le 0$.
\item[(2)]
$\displaystyle\Gamma_r(x)
= \exp\left(
\int_1^\infty (1-u^{-1})^{-r} u^{-x-1} (\log u)^{-1} du\right)$
for $\Real(x)>0$.
\end{itemize}
\end{Thm}

\begin{Ex}
\begin{equation*}
\zeta_{-3}(w;x) = x^{-w} -3(x+1)^{-w} + 3 (x+2)^{-w} -(x+3)^{-w}
\end{equation*}
and
\begin{equation*}
\zeta_{-3}(0;x) = \zeta_{-3}(-1;x) = \zeta_{-3}(-2;x)=0.
\end{equation*}
Notice that $\zeta_{-3}(-3;x) = -6$.
(In general
$\zeta_{-m}(-m;x) = (-1)^m m!$ for integers $m\ge 0$.
\end{Ex}

\begin{Ex}
\begin{equation*}
\zeta_{-\frac12}(w;x) = x^{-w} - \sum_{n=1}^\infty 
\frac{{2n \choose n}}{(2n-1) 4^n} (n+x)^{-w}
\end{equation*}
and 
\begin{equation*}
\zeta_{-\frac12}(0;x) = 1 - \sum_{n=1}^\infty 
\frac{{2n \choose n}}{(2n-1) 4^n} =0,
\end{equation*}
that is
\begin{equation*}
\sum_{n=1}^\infty 
\frac{{2n \choose n}}{(2n-1) 4^n} =1.
\end{equation*}
\end{Ex}
\begin{proof}
The fact (1) follows from the integral representation
\begin{equation*}
\zeta_r(w;x) = \frac{1}{\Gamma(w)} \int_1^\infty
(1-u^{-1})^{-r} u^{-x-1} (\log u)^{w-1} du,
\end{equation*}
since this integral converges for $\Real(w)>-r$
when $\Real(x)>0$,
and $1/\Gamma(w)$ has zeros at
$w=0,-1,\ldots, r+1$. 
Similarly, (2) is seen by looking at $w=0$.
\end{proof}

\section{Proof of Theorem~A}
For a function $N: (1,\infty) \rightarrow \C$ we defined
\begin{equation*}
Z_N(w;s) = \frac{1}{\Gamma(w)} \int_1^\infty N(u) u^{-s-1}(\log u)^{w-1} du
\end{equation*}
and 
\begin{equation*}
\zeta_N(s) = \exp\left(\left.\frac{\partial}{\partial w} Z_N(w;s) \right\vert_{w=0} \right).
\end{equation*}
We calculate these functions in the case of a finite sum
\begin{equation*}
N(u) = \sum_{\alpha} m(\alpha) u^\alpha.
\end{equation*}
It is sufficient to calculate the following monomial case.
\begin{Lem}
Let $N(u)=u^\alpha$, then
\begin{equation*}
Z_N(w;s) = (s-\alpha)^{-w}
\end{equation*}
and
\begin{equation*}
\zeta_N(s) = \frac{1}{s-\alpha}.
\end{equation*}
\end{Lem}
\begin{proof}
\begin{align*}
Z_N(w;s) 
&= \frac{1}{\Gamma(w)} \int_1^\infty u^{\alpha-s-1} (\log u) ^{w-1} du\\
&= \frac{1}{\Gamma(w)} \int_0^\infty e^{-(s-\alpha)t} t^{w-1} dt\\
&= (s-\alpha)^{-w}.
\end{align*}
Hence
\begin{equation*}
\left.\frac{\partial}{\partial w} Z_N(w,s) \right\vert_{w=0}
= - \log(s-\alpha)
\end{equation*}
and
\begin{equation*}
\zeta_N(s) = \frac1{s-\alpha}.
\end{equation*}
\end{proof}

\section{Proof of Theorem~B}
(1) Since 
\begin{align*}
&Z_{N_1\oplus N_2}(w;s) \\
&= \frac{1}{\Gamma(w)} \int_1^\infty (N_1\oplus N_2)(u) u^{-s-1} (\log u)^{w-1} du\\
&= \frac{1}{\Gamma(w)} \int_1^\infty (N_1(u) +  N_2(u)) u^{-s-1} (\log u)^{w-1} du\\
&= Z_{N_1}(w;s) + Z_{N_2}(w;s),
\end{align*}
we have
\[
\zeta_{N_1\oplus N_2}(s) = \zeta_{N_1}(s) \zeta_{N_2}(s).
\]
(2) From
\begin{align*}
(N_1 \oplus N_2)(u) 
&= N_1(u) N_2(u) \\
&= (\sum_{\alpha_1} m_1(\alpha_1) u^{\alpha_1} )
(\sum_{\alpha_2} m_2(\alpha_2) u^{\alpha_2} ) \\
&= \sum_{\alpha_1, \alpha_2} m_1(\alpha_1) m_2(\alpha_2) u^{\alpha_1+\alpha_2},
\end{align*}
we have
\begin{align*}
&Z_{N_1\otimes N_2}(w;s) \\
&= \frac{1}{\Gamma(w)} \int_1^\infty (N_1\otimes N_2)(u) u^{-s-1} (\log u)^{w-1} du\\
&= \frac{1}{\Gamma(w)} \int_1^\infty 
\Big(\sum_{\alpha_1, \alpha_2} m_1(\alpha_1) m_2(\alpha_2) u^{\alpha_1+\alpha_2}\Big) \\
& \hskip4cm \times u^{-s-1} (\log u)^{w-1} du\\
&= \sum_{\alpha_1, \alpha_2} m_1(\alpha_1) m_2(\alpha_2) (s-(\alpha_1+\alpha_2))^{-w}.
\end{align*}
Hence
\[
\zeta_{N_1\otimes N_2}(s) 
= \prod_{\alpha_1,\alpha_2}(s-(\alpha_1+\alpha_2))^{-m_1(\alpha_1) m_2(\alpha_2)}.
\qed
\]

\section{Absolute zeta functions}
\begin{Thm}
Let $r$ be a positive integer.
Then
\begin{itemize}
\item[(1)]
$\displaystyle \zeta_{\Gmr}(s) = \Gamma_{-r}(s-r)$.
\item[(2)]
$\displaystyle \zeta_{\Gmr}(s) 
= \exp\Big( \int_1^\infty N_{\Gmr}(u) u^{-s-1} (\log u)^{-1} du \Big).$
\end{itemize}
\end{Thm}
\begin{proof}
(1) Since
\begin{equation*}
N_{\Gmr}(u) = (u-1)^r,
\end{equation*}
we have
\begin{align*}
&Z_{\Gmr}(w;s)\\
&= \frac1{\Gamma(w)} \int_1^\infty (u-1)^r u^{-s-1}(\log u)^{w-1} du \\
&= \frac1{\Gamma(w)} \int_1^\infty (1-u^{-1})^r u^{-s+r-1} (\log u)^{w-1} du \\
&= \zeta_{-r}(w;s-r).
\end{align*}
Thus, 
\begin{equation*}
\zeta_{\Gmr}(s) = \Gamma_{-r}(s-r).
\end{equation*}
(2) This follows from (1) and Theorem~2(2).
\end{proof}

We notice that Theorem~1 and Theorem~3(1) imply Theorem~C(1)(2).

\section{Functional equations}

\begin{Thm}
Let $r$ be a positive integer.
Then
\begin{equation*}
\frac{\zeta_{\Gmr}(s)}{\zeta_{\Gmr}(r-s)^{(-1)^r}} = S_{-r}(s-r)^{-1}.
\end{equation*}
\end{Thm}
\begin{proof}
From Theorem~3(1), we have
\begin{equation*}
\zeta_{\Gmr}(s) = \Gamma_{-r}(s-r)
\end{equation*}
and 
\begin{equation*}
\zeta_{\Gmr}(r-s) = \Gamma_{-r}(-s).
\end{equation*}
Hence,
\begin{align*}
&\zeta_{\Gmr}(s) \zeta_{\Gmr}(r-s)^{(-1)^{r+1}}\\
&= \Gamma_{-r}(s-r) \Gamma_{-r}(-s)^{(-1)^{r+1}}\\
&= S_{-r}(s-r)^{-1}.
\end{align*}
\end{proof}

We remark that we have the functional equation
\begin{equation*}
\zeta_{\Gmr}(s) = \zeta_{\Gmr}(r-s)^{(-1)^{r}}
\end{equation*}
from Theorem~1(2) and we know that it is
equivalent to $S_{-r}(x)=1$.
Thus we have Theorem~C(3).


\begin{thebibliography}{99}
\bibitem[B]{B} E.W. Barnes, 
On the theory of the multiple gamma functions. 
Trans. Cambridge Philos. Soc. {\bf 19} (1904) 374--425.

\bibitem[CC1]{CC1} 
A. Connes and C. Consani, 
Schemes over ${\mathbb F}_1$ and zeta functions. 
Compositio Mathematica {\bf 146} (2010) 1383--1415. 

\bibitem[CC2]{CC2} 
A. Connes and C. Consani,
Characteristic one, entropy and the absolute point. 
In "Noncommutative Geometry, Arithmetic, and Related Topics, 
Proceedings of the JAMI Conference 2009", Johns Hopkins University Press (2011) 75--140. 

\bibitem[D]{D} A. Deitmar,
Remarks on zeta functions and $K$-theory over ${\mathbf F}_1$, 
Proc. Japan Acad. Ser. A Math. Sci. {\bf 82} (2006) 141--146.

\bibitem[K1]{K1} N. Kurokawa,
Multiple zeta functions: an example.
In ``Zeta Functions in Geometry'' (Tokyo 1990),
Adv. Stud. Pure Math. {\bf 21}, Kinokuniya, Tokyo, 1992, 219--226.

\bibitem[K2]{K2} N. Kurokawa,
Zeta functions over ${\mathbb F}_1$, Proc. Japan Acad. Ser. A Math. Sci. {\bf 81} (2005) 180--184.

\bibitem[KK]{KK} 
N. Kurokawa and S. Koyama, 
Multiple sine functions. Forum Math. {\bf 15} (2003) 839--876.

\bibitem[KO]{KO}
N. Kurokawa and H. Ochiai,
Multiple gamma functions of negative order, 2013, 
preprint.

\bibitem[M]{M} Y. Manin,
Lectures on zeta functions and motives (according to Deninger and Kurokawa),
Ast\'erisque {\bf 228} (1995), 121--163.

\bibitem[S]{Soule} C. Soul\'e,
Les vari\'et\'es sur le corps \`a un \'el\'ement. Mosc. Math. J. {\bf 4} (2004) 217--244.
\end{thebibliography}
\end{document}